\documentclass[11pt]{amsart}

\usepackage{amsmath}
\usepackage{amssymb,amsthm,amstext}
\usepackage{mathtools} 

\usepackage{graphicx}
\usepackage{setspace}
\usepackage{mathrsfs}
\usepackage{bm}
\usepackage{stmaryrd}
\usepackage[left=1.45in,top=1.05in,right=1.45in,bottom=1.0in]{geometry}

\theoremstyle{plain}
\newtheorem{theorem}{Theorem}[section]
\newtheorem*{theorem*}{Theorem}
\newtheorem{proposition}[theorem]{Proposition}
\newtheorem{lemma}[theorem]{Lemma}
\newtheorem{corollary}[theorem]{Corollary}

\newtheorem{theoremintro}{Theorem}

\newtheorem{questionintro}{Question}

\theoremstyle{definition}

\newtheorem{example}[theorem]{Example}

\theoremstyle{remark}
\newtheorem{remark}[theorem]{Remark}
\newtheorem*{remark*}{Remark}
\newtheorem*{example*}{Example}
\newtheorem*{lemma*}{Lemma}

\newcommand{\C}{\mathbb{C}}

\newcommand{\Q}{\mathbb{Q}}
\newcommand{\Z}{\mathbb{Z}}
\newcommand{\F}{\mathbb{F}}
\newcommand{\Fp}{\F_p}
\newcommand{\Fpbar}{\overline{\F}_p}

\newcommand{\Qbar}{\overline{\mathbb{\Q}}}

\DeclareMathOperator{\Spec}{Spec}
\DeclareMathOperator{\Pic}{Pic}
\DeclareMathOperator{\NS}{NS}

\DeclareMathOperator{\Bl}{Bl}
\renewcommand{\P}{\mathbb{P}}

\DeclareMathOperator{\codim}{codim}
\DeclareMathOperator{\Br}{Br}
\DeclareMathOperator{\Db}{\textsf D^{b}}
\DeclareMathOperator{\disc}{disc}

\newcommand{\sheaf}[1]{\mathscr{#1}}
\newcommand{\OO}{\sheaf{O}}
\newcommand{\EE}{\sheaf{E}}
\newcommand{\II}{\sheaf{I}}

\newcommand{\KK}{\mathcal{K}}
\newcommand{\et}{\text{\'et}}
\newcommand{\rk}{\text{rk}}

\DeclareMathOperator{\spe}{sp}
\newcommand{\isom}{\cong}
\newcommand{\Xbar}{\overline{X}}
\newcommand{\Ybar}{\overline{Y}}
\DeclareMathOperator{\Gr}{Gr}
\DeclareMathOperator{\adj}{adj}
\DeclareMathOperator{\tr}{tr}

\usepackage{hyperref}
\hypersetup{pdftitle={Noether--Lefschetz general complete intersection K3 surfaces over Q}}
\hypersetup{pdfauthor={Asher Auel and Henry Scheible}}
\hypersetup{colorlinks=true,linkcolor=blue,anchorcolor=blue,citecolor=blue}

\begin{document}

\title[Noether--Lefschetz general K3 surfaces over $\Q$]{Noether--Lefschetz general\\ complete intersection K3 surfaces over $\Q$}

\author{Asher Auel}
\address{Department of Mathematics, 
Dartmouth College, Kemeny Hall, Hanover, New Hampshire\\\texttt{\it E-mail
address: \tt asher.auel@dartmouth.edu}}

\author{Henry Scheible}
\address{Department of Mathematics, 
Dartmouth College, Kemeny Hall, Hanover, New Hampshire\\\texttt{\it E-mail
address: \tt henry.j.scheible.26@dartmouth.edu}}

\begin{abstract}
We prove that the locus of Noether--Lefschetz general polarized K3 surfaces of degree $d$ defined over $\Q$ is Zariski dense in the moduli space for $d\leq 8$.  Previously, this was proved by van Luijk in the quartic case, and it follows from work of Elsenhans and Jahnel in the degree 2 case.  Innovations on their methods, and employing Mukai's Hodge isogeny, suffices to handle the degree 8 case.  New methods allow us to deal with the case of degree 6.
\end{abstract}

\maketitle

\section*{Introduction}

For each even $d \geq 2$, the moduli space $\KK_d$ of primitively polarized K3 surfaces of degree $d$ is a 19 dimensional quasiprojective variety.   Away from a countably infinite union of divisors in $\KK_d$, a polarized K3 surface of degree $d$ has Picard rank 1 over $\C$. These are the Noether--Lefschetz general K3 surfaces.  We consider the following open question, cf.\ \cite[Ch.~17,~p.~408-409]{huybrechts_lectures_2016}, \cite[Remark~1.3.7]{sosna_derived_2010}, \cite[Problems~1.4--1.5]{noauthor_open_2025}.

\begin{questionintro}
\label{question}
For a given even $d \geq 2$, does there exist a primitively polarized K3 surface $S$ of degree $d$ defined over $\Q$ such that $\Pic(S_\C) \isom \Z$? 
\end{questionintro}

In other words, are there $\Q$-points on the Noether--Lefschetz general locus in $\KK_d$?  Since this locus is the complement of infinitely many divisors, it is not \emph{a priori} clear whether it contains rational points over any given countable field.  However, Ellenberg~\cite{ellenberg_k3_2004} proved that the Noether--Lefschetz general locus in $\KK_d$ does admit $\Qbar$-points for every $d \geq 2$.  Earlier, results of Terasoma~\cite{terasoma_complete_1985} gave a positive answer to Question~\ref{question} for $d=4,6,8$, namely, when the general K3 surface of degree $d$ is a complete intersection.  We remark that a conjecture of Shafarevich~\cite{shafarevich_arithmetic_1996}, which predicts that only finitely many isomorphisms classes of lattices arise as geometric Picard lattices of K3 surfaces defined over a fixed number field, would imply that Question~\ref{question} has a positive answer for only finitely many $d$.  Shafarevich proved the conjecture for K3 surfaces over a number field of geometric Picard rank 20.  Orr and Skorobogatov~\cite{orr_finiteness_2018}, and  Orr~\cite{orr_uniformity_2021}, proved Shafarevich's conjecture for K3 surfaces of CM type.  However, the conjecture seems wide open for K3 surfaces of geometric Picard rank one.

In a similar vein, since $\KK_d$ becomes of general type for $d > 122$ by the work of Gritsenko, Hulek, and Sankaran~\cite[Theorem~6.1]{gritsenko_moduli_2013}, the Bombieri--Lang conjecture would predict that the set of rational points is not Zariski dense, making rational points on the Noether--Lefschetz general locus increasingly rare as $d$ grows.  In contrast, when $d$ is small, $\KK_d$ is unirational and we expect many rational points. Our main result completes the picture for complete intersection K3 surfaces.

\begin{theoremintro}
\label{thm:dense}
 The set of Noether--Lefschetz general K3 surfaces defined over $\Q$ is Zariski dense in $\mathcal{K}_d$ for $d \leq 8$.
\end{theoremintro}

Van Luijk \cite{van_luijk_k3_2007} proved this in degree 4 by a pioneering method of leveraging the Weil and Tate conjectures, together with properties of the specialization homomorphism for the Picard group, to construct specific quartic K3 surfaces over $\F_{p}$ and $\F_{p'}$ with geometric Picard rank 2 and incompatible Picard lattices, forcing a common lift to $\Q$ to have geometric Picard rank 1. Kloosterman~\cite{kloosterman} showed that in van Luijk's method, the specific K3 surface modulo the second prime $p'$ could be traded for information coming from the Artin--Tate formula.  Elsenhans and Jahnel~\cite{elsenhans_jahnel_2008,elsenhans_jahnel:which} adapted van Luijk's strategy for K3 surface of degree 2, and then developed a new general technique~\cite{elsenhans_picard_2011} for K3 surfaces (applying it in degree 2) that only required working modulo a single prime, see also \cite[Proposition~5.3]{hassett-varilly}.  These authors were able to present the first explicit examples of Noether--Lefschetz general K3 surfaces over $\Q$; the case of quartic K3s had been a long-standing challenge due to Mumford, with van Luijk's work yielding the first known explicit examples, see \cite[\S~2.4-2.6]{varilly:K3}.  As already utilized in \cite[p.~12]{van_luijk_k3_2007}, the nature of the explicit constructions, by lifting specific K3 surfaces over a finite field, yields the Zariski density in  Theorem~\ref{thm:dense} for $d=4$, and a similar argument using the construction of Elsenhans and Jahnel~\cite{elsenhans_jahnel_2008,elsenhans_jahnel:which} works for $d=2$, see Section~\ref{sec:density}.  For $d=8$, an argument involving Mukai's degree 8 to degree 2 isogeny (see Section~\ref{sec:8}), together with examples first due to Elsenhans and Jahnel~\cite[\S~8]{elsenhans_picard_2011}, yields the result.  Finally, we must further develop the techniques of van Luijk and Elsenhans--Jahnel for the application to K3 surfaces of degree 6.

A K3 surface of degree 6 is a complete intersection $X=X_{2, 3}\subset \P^4$. We specifically consider such $X$ whose reduction $X_p$ over $\Fp$ contains a line, and then projection from the line yields a double cover $X_p \to \P_{\Fp}^2$.  Generically, this degree 2 model does not contain a tritangent line (see Remark~\ref{rem:tangent_cubic}) so we cannot use the techniques of \cite{elsenhans_picard_2011} directly. 
Instead, we prove a generalization of \cite[Proposition~5.3]{hassett-varilly} to lines contained in projective K3 surfaces, see Section~\ref{sec:6}. After checking $\rho(X_p)=2$, we lift to a surface $X \subset \P^4$ over $\Q$ that does not contain any lines by a verification using Gr\"obner bases.
The first explicit examples were presented by the second author in \cite{spring-poster}.

In degree 8, we use the fact that a complete intersection $X \subset \P^5$ of three quadrics over $\Q$ is Hodge-isogenous to a degree 2 discriminant K3 surface $Y$ over $\Q$, see Section~\ref{sec:8}.  When the reduction $Y_p$ over $\Fp$ admits a tritangent line and has geometric Picard rank two, yet $Y$ does not admit a tritangent line over $\overline{\Q}$, then $Y$ (and hence $X$) will have geometric Picard rank one.  Explicit examples in degree 8 were (implicitly) constructed using this technique in \cite[\S~8]{elsenhans_picard_2011} and \cite[\S5]{mckinnie_brauer_2017}. 

Finally, in Section~\ref{sec:density}, we show how the existence of
explicit constructions of K3 surfaces of geometric Picard rank 1 over
$\Q$ (such as those appearing in work of van Luijk, Elsenhans--Jahnel,
and here) can be used to prove the Zariski density of
the locus of Noether--Lefschetz general K3 surfaces over $\Q$ in the
moduli space.

It remains open whether there exist Noether--Lefschetz general K3 surfaces over $\Q$ of any even degree greater than 8 and in which degrees $d$ is the locus of such surfaces Zariski dense in the moduli space $\KK_d$. 

\smallskip

\noindent \textbf{Acknowledgments.} The authors would like to thank Nick Addington for the idea of projecting from a line in the degree 6 case.  We also thank Jean-Louis Colliot-Th\'el\`ene and Bjorn Poonen for discussions concerning Zariski density of rational points as well as Sarah Frei and Salim Tayou for discussions about the specialization homomorphism.  Peter Doyle and his friend Steve also provided insightful conversations. Auel was partially supported by NSF grant DMS-2200845, and Scheible by an Undergraduate Research Assistantship at Dartmouth (URAD) and a William H.\ Neukom Institute for Computational Science at Dartmouth College research scholarship.

\section{Prerequisites}
\label{sec:prereq}

A \emph{K3 surface} is a smooth, projective, geometrically integral surface $X$ over a field~$k$ whose canonical bundle $\omega_X$ on $X$ is trivial and $H^1(X, \OO_X) = 0$. A \emph{polarized K3 surface} is a pair $(X, H)$ where $X$ is a K3 surface and $H\in \Pic(X)$ is an ample divisor class. We say that the polarization is of \emph{degree} $d$ if $H^2=d$. For each even $d \geq 2$, the moduli space $\KK_d$ of primitively polarized K3 surfaces of degree $d$ is a 19-dimensional quasiprojective variety. For integers $r,e$, we define the \emph{Noether-Lefschetz divisor} $\KK_{d, e}^r \subset \KK_d$ to be the locus of polarized K3 surfaces $(X, H)$ of degree $d$ such that $\NS(X)$ contains a primitive rank 2 sublattice spanned by $H$ and $L$ with Gram matrix
\[
\begin{array}{c|cc}
X & H & L \\\hline
H & d & e \\
L & e & r
\end{array}
\]
When $r$ is even and $dr-e^2<0$ then $\KK_{d, e}^r \subset \KK_d$ is a nonempty irreducible divisor by \cite{ogrady:thesis}, cf.\ \cite{dolgachev:K3}, \cite{greer_li_tian}, \cite{alexeev_engel}.

A subvariety $X\subseteq \P^n$ is a \emph{complete intersection} if the ideal of $X$ is generated by exactly $\codim X$ elements.  
The adjunction formula and the Lefschetz hyperplane theorem imply that the only complete intersection K3 surfaces are the quartic in $\P^3$, the intersection of a quadric and a cubic in $\P^4$, and the intersection of three quadrics in $\P^5$, yielding primitive polarizations of degrees $d=4,6,8$ respectively.

The \emph{Picard rank} of $X$ is the rank $\rho(X) = \rk_\Z \NS(X)$ of the N\'eron--Severi group $\NS(X)$ and the \emph{geometric Picard rank} of $X$ is $\rk_\Z \NS(\overline{X})$, where $\overline{X} = X \times_{\Spec k} \Spec \overline{k}$ for a choice of algebraic closure $\overline{k}$ of $k$. 
For a K3 surface $X$ over $k$, we have that $\Pic (X)\cong \NS (X)\cong \Z^{\rho(X)}$ is an even lattice of signature $(1,\rho(X)-1)$ with respect to the intersection product, with $1\leq \rho(X)\leq 22$, see \cite[Section~17.2]{huybrechts_lectures_2016}. 

Now, let $X$ be a K3 surface over $\Q$ and $p$ be a prime. If there exists a proper flat model of $X$ over $\Z$ with a smooth fiber $X_p$ over $\F_p$, we say that $p$ is a prime of \emph{good reduction} for $X$.  We write $\overline{X}_p = X_p \times_{\Spec \F_p} \Spec \overline{\F}_p$.
For a prime $p$ of good reduction, there exist \emph{specialization homomorphisms}
\[\spe: \NS(X)\to \NS(X_p) \quad \text{and} \quad \overline{\spe}: \NS(\overline{X})\to \NS(\overline{X}_p)\]
which are injective and respect the intersection product, cf.\ \cite[\S20.3]{fulton}, \cite[Proposition~6.2]{van_luijk_elliptic_2007}, \cite[\S3.2]{maulik-poonen}. Moreover, we have the following.

\begin{lemma}[Elsenhans and Jahnel~{\cite[Corollary 3.7]{elsenhans_picard_2011}}]
\label{lemma:cokerneltorsionfree}
Let $X$ be a proper scheme over $\Q$ and $p > 2$ be a prime of good reduction. Then the cokernel of the specialization homomorphism $\overline{\text{sp}}: \NS(\overline{X})\to \NS(\overline{X}_p)$ is torsion free.
\end{lemma}

For a variety $X$ over $\F_p$, we have the \emph{absolute Frobenius} morphism $F: X\to X$ which is the identity on the topological space of $X$ and acts by $u\mapsto u^p$ on the structure sheaf. Let $F^*$ be its pullback on $H^2_{\et}(\overline{X}_p, \Q_\ell)$. We call the characteristic polynomial of $F^*$ the \emph{Weil polynomial} of $X_p$. By the Weil conjectures, it is a polynomial of degree 22 whose roots are algebraic integers with absolute value $p$.

For a K3 surface $X_p$ over $\F_p$, the cycle class map $\NS(\Xbar_p) \otimes_{\Z} \Q_\ell \to H^2_{\et}(\overline{X}_p, \Q_\ell(1))$ is a Frobenius equivariant injection, cf.\ \cite{maulik-poonen}, \cite[Proposition~6.2]{van_luijk_elliptic_2007}.
By observing that each divisor class over $\Fpbar$ is defined over some finite extension $\F_{p^m}$ of $\Fp$, each divisor of $\overline{X}_p$ is fixed by $F^m$ for some positive integer $m$. Thus, if we let $F^*(1)$ be the induced map (of $\Q_\ell$-vector spaces) on $H^2_{\et}(\overline{X}_p, \Q_\ell(1))$, under the cycle class map, each divisor class becomes an eigenvector of $F^*(1)$ with eigenvalue a root of unity. Thus, by the injectivity of the cycle class map, the $\Q_\ell$-dimension of $\NS(\overline{X}_p) \otimes_{\Z_\ell} \Q_\ell$, and hence the $\Z$-rank of $\NS(\Xbar_p)$, is bounded above by the number of eigenvalues of $F^*(1)$ that are roots of unity, counted with multiplicity. Eigenvalues of $F^*(1)$ differ from eigenvalues of $F^*$ by a factor of $p$, so the geometric Picard rank $\rho(\Xbar)$ is bounded above by the number of eigenvalues $\lambda$ of $F^*$ for which $\lambda/p$ is a root of unity, counted with multiplicity, cf.\ \cite[Corollary~2.3]{van_luijk_k3_2007}. In fact, a consequence of the Tate conjecture for K3 surfaces, proved by Nygaard--Ogus~\cite{nygaard-ogus}, Maulik~\cite{maulik:tate}, Charles~\cite{charles2013,charles2016},
Madapusi Pera~\cite{madapusipera:odd}, Kim--Madapusi Pera~\cite{kim_madapusipera,kim_madapusipera:erratum}, and Ito--Ito--Koshikawa~\cite{ito_ito_koshikawa}, is that the upper bound is tight.

\begin{theorem}
    \label{lemma:weilbound}
    For a K3 surface $X$ defined over $\F_q$, the geometric Picard rank $\rho(\Xbar)$ is equal to the number of eigenvalues $\lambda$ of $F^*$ for which $\lambda/q$ is a root of unity, counted with multiplicity.
\end{theorem}

There are algorithms to compute the Weil polynomial of a K3 surface over $\F_p$ given a degree 2 model (due to Elsenhans and Jahnel~\cite{elsenhans_picard_2011} and implemented in \texttt{Magma}~\cite{magma}) or given a model as a hypersurface in toric varieties (due to Kedlaya~\cite{kedlaya}, Abbott--Kedlaya--Roe~\cite{AKR}, and Costa--Harvey--Kedlaya~\cite{CHK}).  Over small finite fields, it can be faster to use specialized point counting techniques, whereas over medium-sized finite fields, $p$-adic methods are usually faster.

\section{Degree 8 Case}
\label{sec:8}

A general polarized K3 surface $(X,H)$ of degree 8 is a complete intersection of three quadrics in $\P^5$.  In this case, the space of quadrics $\P^2 = \P(H^0(\P^5,\II_X(2)))$ containing $X$, called the net of quadrics associated to $X$, is spanned by the three quadrics. The locus $C\subseteq \P^2$ of degenerate quadrics in the net is a sextic curve. Let $f : Y \to \P^2$ be the discriminant double cover of the net, which is branched over $C$.  The discriminant cover can be defined either in terms of the center of the even Clifford algebra, cf.\ \cite[\S3]{kuznetsov:quadrics}, \cite[\S1.6]{ABB:quadrics}, or the Stein factorization of the relative moduli space of maximal isotropic subspaces, cf.\ \cite[\S3]{hassett_varilly_varilly}, \cite[\S3.2]{mckinnie_brauer_2017}, associated to the net.  
These two perspectives are equivalent by \cite[Appendix~B]{ABB:quadrics}. 
When $C$ is smooth, then $Y$ is smooth and $(Y, f^*\OO_{\P^2}(1))$ is a polarized K3 surface of degree 2, called the \emph{discriminant} K3 surface associated to~$X$.  

Explicitly, over a field $k$ of characteristic $\neq 2$, if $X = V(q_0,q_1,q_2) \subset \P^5 = \P(V)$ for quadratic forms $q_0,q_1,q_2$ on $V$, then we consider the \emph{linear span} quadratic form $q = x_0q_0 + x_1q_1 + x_2q_2$ on $V \otimes_k k[x_0,x_1,x_2]$, see \cite[Definition~1.2.3]{ABB:quadrics}.  The signed discriminant $\disc(q)=-\det(x_0Q_0+x_1Q_1+x_2Q_2) \in k[x_0,x_1,x_2]$, with $Q_i$ the symmetric Gram matrix associated to the quadratic form $q_i$, is then a homogeneous form of degree 6.  The discriminant K3 surface $Y \subset \P(1,1,1,3)$ is the variety defined by $y^2 = \disc(q)$ in weighted homogeneous coordinates $(x_0:x_1:x_2:y)$, with the double cover $Y \to \P^2$ given by projection away from the last coordinate, see \cite[\S5]{mckinnie_brauer_2017}.

\begin{proposition}
\label{prop:discriminantcontinuous}
The formation of the discriminant K3 surface 
determines a dominant rational map $\Phi: \KK_8\dashrightarrow \KK_2$. 
\end{proposition}
\begin{proof}
The only question is whether the above explicit construction is defined on an open in $\KK_8$.  The construction is well defined on the locus in $\KK_8$ where the polarization $H$ embeds $X$ as a smooth complete intersection of 3 quadrics whose discriminant K3 is smooth.  This locus is the complement of finitely many Noether--Lefschetz divisors.  Indeed, the failure of smoothness occurs when the polarization fails to be very ample, which occurs on the unigonal locus $\KK_{8, 1}^{0}$ or the hyperelliptic locus $\KK_{8, 2}^0$, see \cite[\S5,~\S8]{saint-donat}; further, the failure of being a complete intersection occurs on the trigonal locus $\KK_{8, 3}^0$, see \cite[Theorem~7.2]{saint-donat}; when $X$ is smooth, the discriminant K3 $Y$ can acquire at most ordinary double points, see \cite[\S6.2]{beauville:prym}, and this occurs on $\KK_{8, 1}^{-2}$ (when $X$ contains a line) or $\KK_{8, 2}^{-2}$ (when $X$ contains a conic). Finally, the rational map is dominant by \cite[Proposition~6.23]{beauville:prym}.
\end{proof}
  
In particular, there exist Zariski open sets $\KK_8^\circ \subset
\KK_8$ and $\KK_2^\circ \subset \KK_2$ such that $\Phi: \KK_8^\circ\to \KK_2^\circ$ is a surjective morphism. 

When a complete intersection K3 surface $X \subset \P^5$ of degree 8 is defined over $\Q$, the explicit construction of the discriminant K3 shows that $Y \to \P^2$ is also defined over~$\Q$.  The following result goes back to Mukai's pioneering work~\cite{mukai:invent,mukai:bombay,mukai:sugaku} on isogenies between K3 surfaces.

\begin{proposition}
\label{lemma:tfmpartnerpreservespic}
For a smooth complete intersection $X \subset \P^5$ of three quadrics over $\Q$ whose associated discriminant $Y \to \P^2$  is smooth, we have that $\rho(\Xbar) = \rho(\overline{Y})$.
\end{proposition}

Since the base change map $\NS(X \times_{\Spec \Q} \Spec \overline{\Q}) \to \NS(X \times_{\Spec \Q} \Spec \C)$ is an isomorphism by rigidity for the N\'eron--Severi group, cf.\ \cite[Proposition~3.1]{maulik-poonen}, to prove Proposition~\ref{lemma:tfmpartnerpreservespic} we are reduced to working over the complex numbers.  

To this end, we briefly review Mukai's theory of isogenies of K3 surfaces and twisted derived equivalence. More details can be found in \cite[Section~16.4]{huybrechts_lectures_2016} and \cite{huybrechts:isogeny}.  Recall the Torelli theorem, which states that K3 surfaces $X$ and $Y$ over $\C$ are isomorphic if and only if there exists a Hodge isometry $H^2(X,\Z) \to H^2(Y,\Z)$.  More generally, a \emph{Hodge isogeny} between K3 surfaces $X$ and $Y$ over $\C$ is a rational Hodge isometry $H^2(X,\Q) \to H^2(Y,\Q)$.  Note that any Hodge isogeny restricts to an isomorphism $\NS(X) \otimes_\Z \Q \isom H^{1,1}(X,\Q) \to H^{1,1}(Y,\Q) \isom \NS(Y) \otimes_\Z \Q$, hence $X$ and $Y$ have the same Picard rank.  One source of Hodge isogenies is considering moduli spaces of sheaves.  The \emph{Mukai lattice} of $X$ is
\[\widetilde{H}(X, \Z) = H^0(X, \Z)\oplus H^2(X, \Z) \oplus H^4(X, \Z)\]
with the Mukai pairing $v.w = v_2.w_2 - v_0.w_4 - v_4.w_0$, where $v_i.w_j$ is the usual cup product pairing in $H^*(X,\Z)$.
For $v \in \widetilde{H}(X,\Z)$, write $M_X(v)$ for the moduli space of stable sheaves $\EE$ on $X$ with Mukai vector 
\[
v(\EE) = \mathrm{ch}(\EE)\mathrm{Td}_X^{1/2} = (\mathrm{rk}(\EE),c_1(\EE),c_1(\EE)^2/2-c_2(\EE) + \mathrm{rk}(\EE)) = v.
\]  
If $v^2 = 0$ then $M_X(v)$ is also a K3 surface, see
\cite[Theorem~1.4]{mukai:bombay}.  If $M_X(v)$ is a fine moduli space
then $X$ and $M_X(v)$ are \emph{Fourier--Mukai partners}, where the
universal sheaf on $X \times M_X(v)$ determines a Fourier--Mukai
equivalence $\Db(X) \to \Db(M_X(v))$ of bounded derived categories.
In general, $M_X(v)$ is not a fine moduli space, and $X$ and $M_X(v)$
are \emph{twisted Fourier--Mukai partners}. Indeed, the universal
sheaf on $X \times M_X(v)$ is $p^*\alpha$-twisted for the class
$\alpha\in \Br(M_X(v))$ that obstructs the fineness of the moduli
space, where $p : X \times M_X(v) \to M_X(v)$ is the projection.  This
universal twisted sheaf determines a twisted Fourier--Mukai
equivalence $\Db(X) \cong \Db(M_X(v),\alpha)$, see
\cite[Theorem~1.3]{caldararu:nonfine}.  Following
Mukai~\cite[Corollary~6.5]{mukai:bombay},
Huybrechts~\cite[Section~1.3]{huybrechts:isogeny} constructs a Hodge isogeny
between $X$ and $M_X(v)$ using Chern characters, in $H^2(X,\Q) \otimes H^2(M_X(v),\Q)$, of the universal twisted sheaf to induce a cohomological correspondence.

Returning to the case where $X \subset \P^5$ is a complete
intersection of three quadrics with associated discriminant $Y \to
\P^2$, Mukai \cite[Example~0.9]{mukai:invent},
\cite[Example~2.2]{mukai:sugaku} shows that $Y = M_X(v)$ is a moduli
space of stable sheaves on $X$ with Mukai vector $v = (2, H, 2)$.
Moreover, the universal twisted sheaf on $X \times Y$ induces an
equivalence $\Db(X) \isom \Db(Y,\alpha)$ for a Brauer class $\alpha
\in \Br(Y)[2]$ that can be defined explicitly in terms of the net of
quadrics, see \cite[Section~3.2,~Lemma~15]{mckinnie_brauer_2017}, and the
induced isogeny $H^2(X,\Q) \to H^2(Y,\Q)$ determines an inclusion $T_X \hookrightarrow T_Y$ of
transcendental lattices with cokernel $\Z/2\Z$. This is Mukai's
isogeny between a degree 8 K3 surface $X$ and its degree 2
discriminant K3 surface $Y$, which yields
Proposition~\ref{lemma:tfmpartnerpreservespic}.

\smallskip

Finally, we give a condition to ensure that a K3 surface of degree 8 over $\Q$ is Noether--Lefschetz general in terms of its discriminant K3 surface.  

\begin{proposition}
\label{thm:disc_rank1}
Let $X\subset \P^5_\Q$ be a smooth complete intersection of three quadrics whose associated discriminant $Y \to \P^2$ is smooth, and let $Y_p \to \P^2_{\F_p}$ be the reduction modulo $p$ for a prime $p > 2$ of good reduction of $Y$. If $\rho(\overline{Y}_p) = 2$, the branch locus of $\overline{Y}_p \to \P^2_{\F_p}$ has a tritangent line, and the branch locus of $\overline{Y}$ has no tritangent line, then $\rho(\Xbar) = \rho(\Ybar) = 1$.
\end{proposition}   
\begin{proof}
We follow the strategy outlined in \cite[Example~1.7]{elsenhans_picard_2011}, \cite[Proposition~5.3]{hassett-varilly}, and \cite[Proposition~26]{mckinnie_brauer_2017} to prove that $\rho(\Ybar)=1$, and then appeal to Proposition~\ref{lemma:tfmpartnerpreservespic}. Indeed, by Proposition~\ref{lemma:tfmpartnerpreservespic} and the injectivity of the specialization map, we have that $\rho(\Xbar)=\rho(\Ybar) \leq 2$.  Assume, to get a contradiction, that $\rho(\Xbar)=\rho(\Ybar) = 2$.  In this case, the image of $\overline{\spe} : \Pic(\Ybar) \to \Pic(\Ybar_p)$ has rank 2, so has torsion cokernel. By Lemma~\ref{lemma:cokerneltorsionfree}, the cokernel is torsion free, so it is 0, hence $\overline{\spe}$ is surjective. 

Now, the pullback of the tritangent line to $\Ybar_p$ splits into two irreducible components. Let $L$ be the divisor class of one of these components, and let $L' \in \Pic(\Ybar)$ be a lift of $L$ under $\overline{\spe}$. As in \cite[Example~1.7]{elsenhans_picard_2011}, $L$ then must arise from a tritangent line to the branch locus of $Y \to \P^2$, but there are none by hypothesis, giving a contradiction.
\end{proof}

To find explicit $X$ over $\Q$ satisfying the hypotheses of Proposition~\ref{thm:disc_rank1} for a fixed prime~$p$, one can first iterate over complete intersections $X_p \subset \P^5_{\F_p}$ and check whether the associated discriminant double cover $Y_p$ admits a tritangent.  We then use the techniques of Section~\ref{sec:prereq} to check whether $\rho(X_p) = 2$.  One can find examples after a number of iterations:

\begin{example}
\label{ex:deg8}
Let $X_{47}=V(q_1, q_2, q_3)\subset \P^5$ over $\F_{47}$, where
\begin{align*}
q_1 &= 5x_0^2 + 6x_0x_1 + x_1^2 + 10x_0x_2 + 45x_1x_2 + x_2^2 + 6x_0x_3 \\ &\quad+ 45x_1x_3 + 45x_2x_3 + 37x_0x_4 + 39x_1x_4 + 39x_2x_4 \\ &\quad+ 2x_4^2 + 10x_0x_5 + 10x_1x_5 + 45x_2x_5
+ 8x_4x_5 + 2x_5^2 \\
q_2 &= 43x_0^2 + 41x_0x_1 + 41x_0x_2 + 8x_1x_2 + 44x_2^2 + 2x_0x_3 \\ &\quad+ 39x_1x_3 + 5x_3^2 + 45x_0x_4 + 2x_1x_4 + 45x_2x_4 + 10x_3x_4 + 3x_4^2 \\ &\quad+ 43x_0x_5 + 6x_1x_5 + 
2x_2x_5 + 2x_3x_5 + 10x_4x_5 + 3x_5^2 \\
q_3 &= 5x_0^2 + 45x_0x_1 + 46x_1^2 + 37x_0x_2 + 2x_1x_2 + 4x_2^2 + 2x_0x_3 \\ &\quad+ 8x_1x_3 + 6x_2x_3 + 42x_3^2 + 8x_0x_4 + 39x_1x_4 + 43x_2x_4 + 4x_3x_4 \\ &\quad+ 2x_1x_5 + 
43x_2x_5 + 43x_3x_5 + 39x_4x_5 + 44x_5^2 
\end{align*}
and $(x_0:\dotsc:x_5)$ are homogeneous coordinates on $\P^5$.  Then $X_{47}$ is a smooth complete intersection whose discriminant $Y_{47} \to \P^2$ is smooth with branch locus defined by the vanishing of
\begin{multline*}
  13u^6 + 43u^5v + 19u^4v^2 + 7u^3v^3 + 46u^2v^4 + 11uv^5 + 21v^6 
+ 43u^5w + 22u^4vw \\+ u^3v^2w + 27u^2v^3w + 17uv^4w + 41v^5w 
+ 26u^4w^2 + 42u^3vw^2 + 33u^2v^2w^2 \\+ 46uv^3w^2 + 19v^4w^2 
+ 42u^3w^3 + 8u^2vw^3 + 34uv^2w^3 \\+ 17v^3w^3 
+ 41u^2w^4 + 32uvw^4 + 21v^2w^4 
+ 46uw^5 + 33vw^5 + 17w^6,
\end{multline*}
where $(u: v: w)$ are homogenous coordinates for $\P^2$.
The discriminant $Y_{47}$ admits the tritangent line $V(u+32y+17w)$ and $\rho(\Xbar_{47}) = \rho(\Ybar_{47}) = 2$, which we verify, using Theorem~\ref{lemma:weilbound}, by computing the Weil polynomial for $Y_{47}$ to be
\begin{multline*}
(t-47)^2(t^{20} - 15t^{19} - 2491t^{18} + 92778t^{17} + 2387929t^{16} - 34157767t^{15} -
6421660196t^{14} \\- 53896076645t^{13} + 27357648505002t^{12} + 95245146647044t^{11} -
49241740816521748t^{10} \\+ 210396528943320196t^{9} + 133496597614536664362t^{8} -
580957415544742891205t^{7} \\- 152907991771376328965156t^{6} -
1796668903313671865340583t^{5} \\+ 277457012068868469490452889t^{4} +
23813049644519406903224087082t^{3}\\ - 1412340634868996252284076212411t^{2} -
18786795237408346374722145844335t \\+ 2766668711962335809450748011342401).
\end{multline*}
\end{example}

With such an example in hand, we can choose random lifts $X$ to $\Z$, then using a Gr\"obner basis algorithm, cf.\ \cite[Algorithm~8]{elsenhans_jahnel_2008},  we can quickly check whether the discriminant double cover $Y$ admits a tritangent line over $\Qbar$.

\begin{corollary}
    Let $X=V(q_1, q_2, q_3)\subset \P^5_{\Q}$, where
    \begin{align*}
                q_1 &= -136x_0^2 - 464x_0x_1 - 140x_1^2 - 272x_0x_2 + 374x_1x_2 + x_2^2 + 288x_0x_3
+ 186x_1x_3 \\ &\quad + 468x_2x_3 + 47x_3^2 - 292x_0x_4 - 196x_1x_4 + 274x_2x_4 -
188x_3x_4 + 237x_4^2 - 84x_0x_5 \\&\quad + 386x_1x_5 + 562x_2x_5 - 282x_3x_5 - 274x_4x_5
- 139x_5^2 \\
        q_2 &= 43x_0^2 + 88x_0x_1 + 141x_1^2 - 100x_0x_2 + 384x_1x_2 + 185x_2^2 -
280x_0x_3 - 8x_1x_3 \\&\quad  - 376x_2x_3 - 89x_3^2 + 562x_0x_4 + 190x_1x_4 + 562x_2x_4 +
104x_3x_4 + 144x_4^2 - 98x_0x_5\\&\quad  - 182x_1x_5 - 468x_2x_5 + 190x_3x_5 - 84x_4x_5 -
44x_5^2 \\
        q_3 &= 193x_0^2 - 2x_0x_1 + 234x_1^2 - 292x_0x_2 + 190x_1x_2 + 51x_2^2 + 2x_0x_3 -
180x_1x_3 \\&\quad + 6x_2x_3 + 183x_3^2 + 8x_0x_4 + 274x_1x_4 + 184x_2x_4 + 286x_3x_4 +
470x_0x_5 - 280x_1x_5\\&\quad  + 560x_2x_5 + 90x_3x_5 - 196x_4x_5 + 185x_5^2
    \end{align*}
Then $X \subset \P^5$ is a smooth complete intersection K3 surface of degree 8 over $\Q$ with geometric Picard rank 1. 
\end{corollary}

We note that Elsenhans and Jahnel~\cite[\S8]{elsenhans_jahnel:which} give an appealing example of a K3 surface $Y \to \P^2$ defined over $\Q$ with geometric Picard rank 1, whose sextic branch curve $C \subset \P^2$ is the determinant of a $6 \times 6$ matrix of linear forms, and whose reductions modulo $5$ are smooth and admit a tritangent line.  From these examples, and using Proposition~\ref{thm:disc_rank1}, one can also extract K3 surfaces of degree 8 over $\Q$ with geometric Picard rank 1.  A similar example can be found in \cite[\S5.4]{mckinnie_brauer_2017}.

\begin{remark}
Note that the locus of polarized K3 surfaces of degree 2 admitting a line tritangent to its branch locus is a Noether--Lefschetz divisor $\KK_{2,1}^{-2} \subset \KK_2$. In particular, by Proposition~\ref{prop:discriminantcontinuous}, we see that the locus in $\KK_8$, consisting of polarized K3 surfaces $X$ of degree 8 whose projective model is a smooth complete intersection of three quadrics in $\P^5$ and whose discriminant $Y \to \P^2$ is smooth and admits a tritangent line, is Zariski closed in $\KK_8^\circ$. We will use this in Section~\ref{sec:density} in the proof that the set of Noether--Lefschetz general K3 surfaces is Zariski dense in $\KK_8$.
\end{remark}

\section{Degree 6 Case}
\label{sec:6}

Every basepoint free nonhyperelliptic polarized K3 surface $(X,H)$ of degree 6 is a complete intersection of type $(2,3)$ in $\P^4$, i.e., a complete intersection of a quadric and cubic hypersurface, see \cite[Theorem~6.1]{saint-donat}.  To bound the geometric Picard rank of $X$, one could follow the strategy of van Luijk~\cite{van_luijk_k3_2007}, computing the Weil polynomial of a reduction $X_p$ at a prime of good reduction and applying Theorem~\ref{lemma:weilbound} in view of the injectivity of the specialization map.  To compute the Weil polynomial, one could
hope to count points over $\F_{p^n}$ for $n=1,\dotsc,12$.  However, a naive point counting algorithm that enumerates over all points in $\P^4$ would be on the border of computational feasibility.  In theory, there exist $p$-adic algorithms to compute the Weil polynomial of a complete intersection, which employ the ``Cayley trick'' to transform a complete intersection to a hypersurface and then use \cite{CHK}, but these are not yet publicly implemented. Instead, in the spirit of \cite[\S3]{addington_auel}, we leverage the geometry of linear projections to design a fast specialized algorithm.  

We first recall that any big and nef divisor $L$ on a K3 surface with $L^2 = 2$ determines a morphism $X \to \P^2$, which is a blow-down followed by a double cover as long as $(X,H)$ is not unigonal, i.e., there exists no divisor $E$ with $L.E=1$ and $E^2=0$.  In the later case, $(L-2E)^2 = -2$ and $L.(L-2E)=0$, hence $L$ is not ample and has Zariski decomposition $L = 2E + (L-2E)$ with moving part $2E$.  In particular, the map $X \to \P^2$ induced by $L$ factors through the elliptic fibration $X \to \P^1$ induced by $E$ and a Veronese embedding $\P^1 \subset \P^2$, so that $X \to \P^2$ is not dominant, see \cite[Section~5,~Scholium]{mayer}.  Conversely, when the map $X \to \P^2$ induced by $L$ is dominant, then any $(-2)$-curve $C \subset X$ with $L.C=0$ is blown down, and the map $X \to \P^2$ factors as a sequence of blow-downs $X \to X_0$ followed by a double cover $X_0 \to \P^2$, see \cite[\S~5]{saint-donat} or \cite[III.3]{BHPV}.  By Riemann--Hurwitz (and assuming we are in characteristic $\neq 2$), the branch locus of this double cover must be a sextic curve, which is smooth if and only if there are no blown-down curves.

With this in mind, we recall the following simplification of the geometry of a complete intersection surface in $\P^4$ that contains a line.  

\begin{lemma} 
\label{lem:line_complete_intersection_case}
Let $X \subset \P^4$ be a nondegenerate $(a,b)$ complete intersection surface with hyperplane section $H$.  If the smooth locus of $X$ contains a line $L$, then projection from $L$ restricts to a dominant generically finite morphism $X \to \P^2$ of generic degree $(a-1)(b-1)$ induced by the linear system of $H-L$. 
\end{lemma}

\begin{proof}
Resolving the projection from $L$ yields a morphism $\phi : \Bl_{L}\P^4 \to \P^2$.  As $L \subset X$ is a smooth divisor in the smooth locus, the strict transform coincides with $X$, giving an embedding $X \subset \Bl_L \P^4$ and we consider the restriction $\phi|_X : X \to \P^2$.   By the construction of the projection map in terms of planes through $L$, the projective morphism $\phi|_X$ corresponds to the the linear system of $H-L$.  The fibers of the projection are the residual intersections of $X$ with planes containing $L$.  Writing $X=X_a\cap X_b$ as a complete intersection of hypersurfaces in $\P^4$ of degree $a$ and $b$, the residual intersections of $X_a$ and $X_b$ with a plane $P$ containing $L$ are plane curves $C_{a-1}, C_{b-1} \subset P$ of degree $a-1$ and $b-1$, which are both positive since $X$ is nondegenerate.  Hence the fiber of $\phi|_X : X \to \P^2$ corresponding to $P$ is the intersection $C_{a-1} \cap C_{b-1}$ of these plane curves.  Over the generic fiber, this intersection of plane curves must be finite (otherwise they would share a component and hence $X$ would be birational to a curve bundle over $\P^2$, which is impossible since $X$ is a surface), and hence by B\'ezout's theorem is a finite scheme of degree $(a-1)(b-1)$.
\end{proof}

In the case of a sextic K3 surface, we can say more.

\begin{lemma}
\label{lem:sextic_k3_decomp}
Let $X \subset \P^4$ be a smooth sextic K3 surface containing a line $L$.  The projection $X \to \P^2$ from $L$ is the composition of a blow-down of exceptional curves and a finite flat double cover branched over a sextic curve.
\end{lemma}
\begin{proof}
By Lemma~\ref{lem:line_complete_intersection_case}, projection $X \to \P^2$ from $L$ is dominant and induced by the linear system of $H-L$.  Since $(H-L)^2 = 6 - 2\cdot 1 -2 = 2$ and $H-L$ induces a dominant map $X \to \P^2$, then as explained earlier, it must blow-down $(-2)$-curves $C$ such that $(H-L).C=0$ followed by a double cover, which must be branched over a sextic curve.
\end{proof}

\begin{remark}
\label{rem:6line}
We remark that in the moduli space $\KK_6$, the locus of polarized K3 surfaces $(X,H)$  of degree 6 ``containing a line'' i.e., an irreducible curve $L \subset X$ with $H.L=1$ and $L^2=-2$, is the Noether--Lefschetz divisor $\KK_{6, 1}^{-2}$.  Then $H-L$ is effective and $(H-L)^2=2$.  Moreover, $H-L$ is base point free (and the induced map $X \to \P^2$ is a double cover) if and only if $(X,H)$ is neither unigonal nor hyperelliptic, i.e., is not contained in the Noether--Lefschetz divisors $\KK_{6,1}^0$ nor $\KK_{6,2}^0$. 
\end{remark}

As a consequence, a degree 6 K3 surface containing a line admits a degree 2 model, on which we can use highly optimized Weil polynomial calculation algorithms implemented in \texttt{Magma}~\cite{magma} by Elsenhans and Jahnel.

However, to employ this algorithm, we need to be able to computationally verify that projection from the line is a double cover $X \to \P^2$ and we need to compute an explicit model as a hypersurface of weighted degree 6 in $\P(1,1,1,3)$, or, what will be enough for our purposes (see Remark~\ref{rem:twist}), at least explicitly compute the sextic branch curve in $\P^2$.  To this end, up to a linear change of variables, we can fix $L = V(x_0,x_1,x_2)$ where $(x_0:x_1:x_2:y_0:y_1)$ are homogenous coordinates for $\P^4$.  Assuming that $X=V(f_2, f_3) \subset \P^4$ contains $L$, where $f_2,f_3 \in k[x_0,x_1,x_2,y_1,y_2]$ are homogenous forms of degree 2 and 3, respectively, then $f_2$, $f_3$ can be expressed as 
\begin{equation}
\begin{split}
\label{eq:containline}
f_2 &= l_0y_0 + l_1y_1 + q \\
f_3 &= l_{00}y_0^2 + l_{01}y_{0}y_{1}+l_{11}y_1^2 + q_0y_0 + q_1y_1 + c
\end{split}
\end{equation}
where $l_0, l_1, l_{00}, l_{01}$, and $l_{11}$ are homogenous linear forms, $q, q_0, q_1$ are homogenous quadratic forms, and $c$ is a homogenous cubic form, all in $k[x_0, x_1, x_2]$. 

\begin{theorem}
\label{thm:dominant}
Let $k$ be a field of characteristic $\neq 2$ and let $X\subset\P^4$ be a smooth sextic K3 surface over $k$ containing a line $L \subset \P^4$.  In the notation of \eqref{eq:containline}, consider the following symmetric matrix
\begin{equation}
\label{eq:4mat}
A = 
\begin{pmatrix}
0 & v^t \\
v & M
\end{pmatrix}
=
\begin{pmatrix}
0   & l_0     & l_1     & q  \\
l_0 & 2l_{00} & l_{01}  & q_0 \\
l_1 & l_{01}  & 2l_{11} & q_1 \\
q   &q_0      & q_1     & 2c
\end{pmatrix}    
\end{equation}
of homogeneous forms on $\P^2$, write $g_6 = \det(A)$ and $g_3 =
\det(M)$, and let $D=V(g_6)$ and $C = V(g_3)$ in $\P^2$.  If $D$ is
smooth, then projection from $L$ determines a finite flat double cover
$\pi : X \to \P^2$ with branch locus $D$ that is tangent to the nodal cubic curve $C \subset \P^2$.  

For a generic choice of coefficients, $X \subset \P^4$ is smooth, $D \subset \P^2$ is smooth, $C \subset \P^2$ has a single node away from $D$, and $C$ intersects $D$ tangentially in 9 distinct points.
\end{theorem}

\begin{remark}
\label{rem:tangent_cubic}
Before we give the proof, we remark that Hodge theory predicts the existence of a $9$-tangent rational cubic, i.e., a rational cubic $C \subset \P^2$ tangent to the branch sextic $D$ in 9 points, in the double cover model of $X$. Indeed, by Lemma~\ref{lem:line_complete_intersection_case}, in a polarized K3 surface $(X,H)$ of degree 6 containing a line $L$, 
the double plane model (assuming it exists) is determined by the linear system $H-L$, and we obtain lattice polarizations of degree 6 and 2 on $X$
\[
\begin{array}{c|cc}
X & H & L \\\hline
H & 6 & 1 \\
L & 1 & -2
\end{array}
\qquad
\begin{array}{c|cc}
X   & H-L & L \\\hline
H-L & 2   & 3 \\
L   & 3   & -2
\end{array}
\]
In terms of moduli spaces, projection from $L$ determines a birational map of Noether--Lefschetz divisors $\KK_{6,1}^{-2} \dashrightarrow \KK_{2,3}^{-2}$.  Finally, 
we see that in the double cover model, the image of $L$ in $\P^2$ is a rational cubic curve that splits into two rational curve components $L$ and $3(H-L)-L = 3H-4L$ meeting in $L.(3H-4L)=11$ points on $X$.  The 11 points comprise the preimages of the 9 points of intersection of $C$ and $D$ together with the two distinct preimages of the node on $C$.
\end{remark}

To prove Theorem~\ref{thm:dominant}, we'll need the following
consequence of projective duality.  Recall the matrix adjugate
$\adj(M)$ of a square matrix $M$.

\begin{lemma}
\label{lem:duality}
Let $k$ be a field of characteristic $\neq 2$ and $L \subset \P^n$ a
hyperplane dual to $v \in k^{n+1}$. For a smooth quadric $Q \subset
\P^n$ with Gram matrix $M$, we have that $L$ is tangent to $Q$ if and
only if $v^t \adj(M) v = 0$.

When $n=2$ and $Q$ is a union of two distinct lines in $\P^2$, then $v^t \adj(M) v = 0$ and $v^t M v = 0$ if and only if $L$ is one of the two lines.
\end{lemma}

And we'll also need the following linear algebra identity.

\begin{lemma}
\label{lem:adj}
For an $n \times n$ matrix $M$ and a column vector $v$ of length $n$
over an integral domain, we have
\[
v^t \adj(M) v = -\det
\begin{pmatrix}
0 & v^t \\
v & M
\end{pmatrix}
\]
\end{lemma}

\begin{proof}
Recall that $\adj(M)$ is given by $\adj(M)_{ij} = (-1)^{i+j}M_{ji}$, where $M_{ij}$ is the $(i,j)$-cofactor of $M$, i.e., the determinant of $M$ with the $i$th row and $j$th column removed. Computing the determinant on the right hand side using cofactor expansion along the first row and then the first column, we get 
\begin{align*}
        -\det
    \begin{pmatrix}
    0 & v^t \\
    v & M
    \end{pmatrix} &=-\sum_{i=1}^n (-1)^{i}v_i \sum_{j=1}^n (-1)^{j+1}v_j M_{ji} \\
    &=\sum_{i=1}^n \sum_{j=1}^n (-1)^{i+j} v_iM_{ji} v_j
    = \sum_{i=1}^n \sum_{j=1}^n v_i\adj(M)_{ij} v_j
    = v^t\adj(M) v
    \end{align*}
where $v^t = (v_1,\dotsc,v_n)$.
\end{proof}

\begin{proof}[Proof of Theorem~\ref{thm:dominant}]
As in the proof of Lemma~\ref{lem:line_complete_intersection_case}, resolving the projection from $L$ yields a morphism $\phi : \widetilde{\P}^4 = \Bl_{L}\P^4 \to \P^2$.  Writing $X = X_2 \cap X_3$ for hypersurfaces $X_2,X_3 \subset \P^4$ of degrees 2 and 3, respectively, we consider the strict transforms $\widetilde{X}_2, \widetilde{X}_3 \subset \widetilde{\P}^4$.  We have that $\widetilde{\P}^4 \isom \P(\EE)$ where $\EE = \OO_{\P^2}^{\oplus 2} \oplus \OO_{\P^2}(-1)$, cf.\ \cite[Section~9.3.2]{3264}, and $\phi : \P(\EE) \to \P^2$ coincides with the projective bundle map.  The homogeneous coordinates $(y_0:y_1)$ correspond, via the identification $\widetilde{\P}^4 \isom \P(\EE)$, to a basis of global sections of $\OO_{\P(\EE)}(1)$.  Let $z$ be a nonzero global section
of $\OO_{\P(\EE)}(1)\otimes\phi^*\OO_{\P^2}(-1)$.  Then $z$ is unique up to scaling, as
\begin{align*}
h^0(\P(\EE),\OO_{\P(\EE)}(1)\otimes\phi^*\OO_{\P^2}(-1)) & {} = 
h^0(\P^2,\phi_*\OO_{\P(\EE)}(1)\otimes\OO_{\P^2}(-1))) \\
&{} = 
h^0(\P^2,\EE^\vee \otimes \OO_{\P^2}(-1)) = 1
\end{align*}
by the projection formula.  Thus $(y_0:y_1:z)$ forms a relative
system of homogeneous coordinates on $\P(\EE)$ over $\P^2$.  Then
$\widetilde{X}_2$ and $\widetilde{X}_3$ can be identified with the subschemes of $\P(\EE)$ defined
by the vanishing of global sections
\[
l_0y_0 + l_1y_1 + qz \quad \text{and} \quad
l_{00}y_0^2 + l_{01}y_{0}y_{1}+l_{11}y_1^2 + q_0y_0z + q_1y_1z + cz^2,
\]
of $\OO_{\P(\EE)}(2)\otimes\phi^*\OO_{\P^2}(1)$ and $\OO_{\P(\EE)}(3)\otimes\phi^*\OO_{\P^2}(1)$, respectively. 

This gives an alternate way of considering the proof of Lemma~\ref{lem:line_complete_intersection_case}.  Indeed, restricting $\phi|_{\widetilde{X}_i} : \widetilde{X}_i \to \P^2$ yields a relative plane curve of degree $i-1$ in $\P(\EE)$ determined by one of the above equations.  
Because $L$ is a smooth divisor in the smooth locus of $X$, we know that $X$ coincides with its strict transform in $\widetilde{\P}^4$, and hence that $\widetilde{X}_2\cap \widetilde{X}_3\isom X$. Thus the fiber of $\phi|_X : X \to \P^2$ over a point $a \in \P^2$ is the intersection of the fibers over $a$ of $\phi_{\widetilde{X}_2} : \widetilde{X}_2 \to \P^2$ and $\phi_{\widetilde{X}_2} : \widetilde{X}_2 \to \P^2$.  When both relative hypersurfaces are flat at $a \in \P^2$, equivalently, the homogeneous forms
\[
l_0(a)y_0 + l_1(a)y_1 + q(a)z \quad \text{and} \quad
l_{00}(a)y_0^2 + l_{01}(a)y_{0}y_{1}+l_{11}(a)y_1^2 + q_0(a)y_0z + q_1(a)y_1z + c(a)z^2,
\]
in $(y_0:y_1:z)$ are not identically zero, then the fiber is the intersection of a line and a conic in the fiber over $a$ of $\P(\EE) \to \P^2$.  Even when one relative hypersurface (or both) is not flat, the fiber above $a$ is still nonempty.  In particular, the morphism $\phi|_X : X \to \P^2$ is surjective, and is a double cover away from the non-flat fibers of the individual relative hypersurfaces.  

We note that $\phi_{\widetilde{X}_2} : \widetilde{X}_2 \to \P^2$ is not flat over points in $V(l_0,l_1,q) \subset \P^2$ and $\phi_{\widetilde{X}_3} : \widetilde{X}_3 \to \P^2$ 
is not flat over points in $V(l_{00},l_{01},l_{11},q_0,q_1,c) \subset \P^2$.  Clearly, for generic coefficients, both nonflat loci are empty.  

Over the locus where both $\widetilde{X}_2$ and $\widetilde{X}_3$ are flat, the double cover $\phi|_X : X \to \P^2$ is ramified precisely when the fiber of $\widetilde{X}_2$ (which is a line in the fibral plane) is tangent to the fiber of $\widetilde{X}_3$ (which is a conic).  Lemma~\ref{lem:duality} thus shows that this occurs over the vanishing locus of $v^t\adj(M)v$, where $M$ is the bottom right $3 \times 3$ submatrix in \eqref{eq:4mat} and $v$ is the vector with $v^t = (l_0, l_1, q)$.  By Lemma~\ref{lem:adj}, this is the same as the vanishing locus $D \subset \P^2$ of the determinant of the matrix $A$ in \eqref{eq:4mat}.

Now, we proceed to show that if $D$ is smooth then the map $\phi|_X : X \to \P^2$ is flat, hence a finite flat double cover of $\P^2$ branched over $D$.  By Jacobi's formula, we have that
\[
\frac{\partial}{\partial x_i} \det(A) = \tr\left( \adj(A) \frac{\partial A}{\partial x_i}\right)
\]
and we'll proceed to use the jacobian criterion to show that $D$ is singular at points over which $\phi|_X : X \to \P^2$ is not flat, equivalently, either one of the $\widetilde{X}_i \to \P^2$ is not flat or the fiber of $\widetilde{X}_2$ (the line) is a component of the fiber of $\widetilde{X}_3$ (which is hence a union of lines).  Over a point $a \in \P^2$ where $\widetilde{X}_2$ is not flat, the first rown and column of $A$ is uniformly zero, hence $\adj(A)$ is at most nonzero only in the $(1,1)$-entry.  But this multiplied by $\frac{\partial A}{\partial x_i}$ is the zero matrix.  Similary, where $\widetilde{X}_3$ is not flat, then the entire lower right $3\times 3$ block of $A$ is zero, meaning that already $\adj(A)$ is the zero matrix. Over a point $a \in \P^2$ (which we can assume is a geometric point since we are testing smoothness) where $\widetilde{X}_3$ is a union of distinct lines, one of which is $\widetilde{X}_2$, we can change variables in the fibral plane with homogeneous coordinate $(x:y:z)$ so that $\widetilde{X}_3 = V(xy)$ and $\widetilde{X}_2 = V(x)$.  In these variables,  Lemma~\ref{lem:duality} says that $v^t = (1,0,0)$ or $v^t = (0,1,0)$, in which case $\adj(A)$ is the zero matrix.  Where $\widetilde{X}_3$ is the double line $\widetilde{X}_2$, we can change variables in the fibral plane with homogeneous coordinate $(x:y:z)$ so that $\widetilde{X}_3 = V(x^2)$ and $\widetilde{X}_2 = V(x)$ and $v^t = (1,0,0)$, in which case $\adj(A)$ is again the zero matrix.  Hence in all cases, $D$ is singular at $a$.  This shows that the smoothness of $D$ implies the flatness of the double cover, and it has already been noted that the branch locus will then be $D$.

As for the tangency of the intersection of $D$ and $C$, this is a standard fact about the resolution of $9$ points in $\P^2$ and the shape of the matrix $A$ in \eqref{eq:4mat}, cf.\ \cite[Section~4.1]{beauville:det}.  The fact that $C$ has a node can be checked directly.

Finally, in the space of all coefficients, the required smoothness and intersection conditions are open, so it suffices to find a single case where they all occur simultaneously, e.g., the example appearing in Corollary~\ref{thm:explicitdeg6}.
\end{proof}

\begin{remark}
The fact that the branch locus of the degree 2 model of a sextic K3 surface containing a line has a symmetric determinental presentation of the form $\det A$ where $A$ is the $4 \times 4$ graded-homogeneous matrix of homogeneous forms on $\P^2$ in \eqref{eq:4mat} 
is indicative that there should be a relationship with cubic fourfolds containing a plane.  Indeed, the matrix $A$ determines a cubic fourfold $Z \subset \P^5$ containing a plane, which (in the notation of \eqref{eq:containline}) has equation $f_3 + y_3 f_2 = 0$ if $(x_0:x_1:x_2:y_1:y_2:y_3)$ are homogeneous coordinates on $\P^5$.  Hence $Z$ is singular at $(0:0:0:0:0:1)$ and generically this singularity is an ordinary double point.  As in \cite[Section~4.2]{hassett:special}, the sextic K3 surface $X$ appears in the resolution of the projection $Z \dashrightarrow \P^4$ from the node as the locus in $\P^4$ of lines through the node.  In this case, the tangent cone of $Z$ at the node coincides with the projective cone in $\P^5$ over the quadric $V(f_2) \subset \P^4$.  Note that $Z$ also contains the plane $V(x_0,x_1,x_2) \subset \P^5$.  Hence the locus of these cubic fourfolds forms a component of $\mathcal{C}_6 \cap \mathcal{C}_8$, where $\mathcal{C}_6, \mathcal{C}_8 \subset \mathcal{C}$ are the Hassett divisors parameterizing cubic fourfolds admitting a node and containing a plane, respectively, in the moduli space $\mathcal{C}$ of (semi-stable) cubic fourfolds, see \cite{hassett:special}.  However, this is a different component in $\mathcal{C}_6 \cap \mathcal{C}_8$ from the one considered in \cite{stellari:singular_plane}, since the sextic branch divisor of the associated K3 surface of degree 2 is actually smooth for the very general element.  We record that the intersection form on the lattice of integral type $(2,2)$ Hodge classes in the limiting Hodge structure associated to the nodal cubic fourfold $Z$ has a lattice polarization by
\[
\begin{array}{c|ccc}
Z   & h^2 & P   & X  \\\hline
h^2 & 3   & 1   & 6  \\
P   & 1   & 3   & 1  \\
X   & 6   & 1   & 14
\end{array}
\]
where $h^2$ is the square of the hyperplane class, $P$ is the class of the plane, and $X$ is the class of the sextic K3 surface.
\end{remark}

Now, we proceed to use the above to establish sufficient criteria for a sextic K3 surface over $\Q$ to have geometric Picard rank 1.  The following general criterion for a projective K3 surface to have geometric Picard rank 1 is inspired by \cite[Example~1.7]{elsenhans_picard_2011}, \cite[Proposition~5.3]{hassett-varilly}, and \cite[Proposition~26]{mckinnie_brauer_2017}. 

\begin{theorem}
\label{thm:genericdeg6}
Let $X\subset \P^n_\Q$ be a K3 surface. Assume that $X$ contains no lines over $\overline{\Q}$, and that for some prime $p > 2$ of good reduction $X_p$ contains a line and $\rho(\overline{X}_p)=2$. Then $\rho(\overline{X})=1$.
\end{theorem}

\begin{proof}
Let $H\in \NS(X)$ be the hyperplane section. Assume, to get a contradiction, that $X$ has geometric Picard rank at least 2. If $\rho(\overline{X})>2$, we have a contradiction because $\overline{\textrm{sp}} : \NS(\overline{X}) \to \NS(\overline{X}_p)$ is injective and we are assuming $\rho(\overline{X}_p)=2$. Thus, we can assume that $\rho(\overline{X})=2$. Because $\overline{\textrm{sp}}$ is injective, the image of $\overline{\textrm{sp}}$ has rank 2 as well, so the cokernel has rank 0.  However, by Lemma \ref{lemma:cokerneltorsionfree}, the cokernel is torsion free, so it must be 0.  Thus $\overline{\textrm{sp}}$ is surjective. 

Finally, let $L$ be a line on $X_p$. Because $\overline{\textrm{sp}}$ is surjective, we have a divisor class $L'\in \NS(\overline{X})$ such that $\overline{\textrm{sp}}(L')=L$, but $\overline{\spe}$ preserves degree (i.e., intersection with $H$), so $L'$ is the class of a line on $\overline{X}$, contradicting our hypothesis.
\end{proof}

We remark that for a K3 surface that is a double cover $X \to \P^2$, a
``line'' on $X$, i.e., a divisor class $L \in \Pic(X)$ with $H.L=1$
and $L^2=-2$, is precisely a component of the preimage of a tritangent
line to the branch divisor of the cover.  Hence, one can view
Theorem~\ref{thm:genericdeg6} as the starting point for an amply polarized generalization to any degree of the strategy initiated in \cite[Example~1.7]{elsenhans_picard_2011}, cf.\ \cite[Proposition~5.3]{hassett-varilly}, for degree 2 K3 surfaces.

\medskip

To construct an explicit example, we randomly generate sextic K3 surfaces $X_p$ over $\F_p$ containing a fixed line $L \subset \P^4$ until we find one with $\rho(\Xbar_p)=2$, which we verify by calculating the Weil polynomial using projection from the line.  Then, we lift $X_p$ to $X$ over $\Q$ in such a way that $\overline{X}$ contains no lines, and then apply Theorem~\ref{thm:genericdeg6}.  We implemented this strategy to find the following.
\begin{corollary}
    \label{thm:explicitdeg6}
    Let $X=V(f_2, f_3)\subset \P^4_\Q$ where
    \begin{align*}
        f_2 &= x_0^2  -3x_0x_1 + 3x_1^2 + 5x_0x_2 + 4x_1x_2 + 5x_2^2  -x_0x_3  -2x_1x_3  \\&\quad -3x_2x_3  -5x_0x_4 + 5x_1x_4+47x_3^2+47x_4^2 \\
        f_3 &= 2x_0^3 + 3x_0^2x_1 + 3x_0x_1^2 + x_1^3  -x_0x_1x_2  -3x_1^2x_2 + 4x_0x_2^2  -4x_1x_2^2 + 5x_2^3 \\
        &\quad+ 4x_0^2x_3 + x_0x_1x_3   + 5x_1^2x_3 + 4x_0x_2x_3 + 4x_1x_2x_3 
        -3x_2^2x_3 + 4x_1x_3^2  -x_2x_3^2 \\ &\quad+ 5x_0^2x_4   -4x_0x_1x_4 + 2x_1^2x_4  + x_0x_2x_4 + 4x_1x_2x_4  -2x_2^2x_4 \\ &\quad+ 4x_0x_3x_4  -3x_2x_3x_4  -x_0x_4^2  
        -x_1x_4^2 + 5x_2x_4^2,
    \end{align*}
    Then $X$ is a sextic K3 surface with geometric Picard rank 1.
\end{corollary}
\begin{proof}
We constructed the example so that $X_{47}$ contains the line $V(x_0,x_1,x_2)$. Projecting from this line, and using Theorem~\ref{thm:dominant}, we find that the double cover $X_{47} \to \P^2$ has branch divisor the vanishing of $g_6(u,v,w)$ given by
\begin{multline*}
  14u^6 + 36u^5v + 40u^4v^2 + 2u^3v^3 + 38u^2v^4 
+ 40uv^5 + 26v^6 + 7u^5w + 29u^4vw \\+ 12u^3v^2w 
 + 29u^2v^3w + 2uv^4w + 28v^5w + 29u^4w^2 + 15u^3vw^2 
 + 12u^2v^2w^2 \\+ 16uv^3w^2 + 11v^4w^2 + 40u^3w^3 + 31u^2vw^3 
 + 38uv^2w^3 + 26v^3w^3 + 35u^2w^4 \\+ 10uvw^4 + 18v^2w^4 
 + 2uw^5 + 43vw^5 + w^6
\end{multline*}
where $(u: v: w)$ are homogenous coordinates for $\P^2$. Using \texttt{Magma}'s algorithm, we computed the Weil polynomial of the double cover model $s^2=g_6(u,v,w)$ as
\begin{multline*}
(t - 47)^2(t^{20} + 35t^{19} + 1410t^{18} + 79524t^{17} - 311469t^{16} + 39037448t^{15}
+ 5504280168t^{14} \\- 86233722632t^{13} - 1013246240926t^{12} - 666716026529308t^{11}
- 78339133117193690t^{10} \\- 1472775702603241372t^{9} - 4944318430168024606t^{8}
- 929531864871588625928t^{7} \\+ 131063992946893996255848t^{6}
+ 2053335889501339274674952t^{5} \\- 36190045052461104716146029t^{4}
+ 20411185409588063059906360356t^{3} \\+ 799438095208865803179665780610t^{2}
+ 43835855553952808207685006970115t \\+ 2766668711962335809450748011342401)
\end{multline*}
As the Weil polynomial contains no cyclotomic roots besides $(t-47)^2$ coming from the hyperplane section and the line, Theorem~\ref{lemma:weilbound} implies that $X_p$ has geometric Picard rank 2. 
        
Next, we use a Gr\"obner basis calculation on each Schubert cell of the Grassmannian $\Gr(2, 5)$ of lines in $\P^4$ to verify that $X$ contains no lines over $\overline{\Q}$, cf.\ \cite{elsenhans_picard_2011}. Finally, we conclude that so $\rho(X)=1$ by Theorem~\ref{thm:genericdeg6}. 
\end{proof}

\begin{remark}
\label{rem:twist}
Knowing the branch locus $D\subseteq \P^2$ only specifies the polynomial $g_6$ up to multiplication by a unit. This is insufficient to specify the isomorphism class of the double cover $X\to \P^2$ with branch locus $V(g_6)$, as there can be nontrivial quadratic twists. Explicitly, the double cover models $s^2=g_6$ and $s^2=\lambda g_6$ might be different for $\lambda$ a nonsquare in the base field. In fact, one easily finds examples where these quadratic twists have different point counts over a finite field, hence are be nonisomorphic.  Of course, over a finite field of characteristic $\neq 2$, and for a generic choice of $g_6$, there will be at most one quadratic twist. However, since any two quadratic twists are isomorphic over the algebraic closure, they have the same geometric Picard rank.  So, while we do not know which twist gives the double plane model of our original sextic K3 surface, this is not an issue for our argument. After some experimentation, however, we believe that the correct model for the double cover is $s^2=g_6$, where $g_6$ is as in Theorem~\ref{thm:dominant}.
\end{remark}

\section{Zariski density}
\label{sec:density}

In this section, we will prove Theorem~\ref{thm:dense}, and along the way streamline the argument that van Luijk employs~\cite[Proof~of~Theorem~1.1]{van_luijk_k3_2007} to prove Zariski density in degree~4.  Note that, unlike in van Luijk, we are not concerned with whether our Noether--Lefschetz general K3 surfaces have infinitely many $\Q$-points or not, though this could make for an interesting follow up.

In degrees $d \leq 8$, one constructs a subset of $\KK_d(\Q)$ consisting of Noether--Lefschetz general K3 surfaces.  In each case, this subset has the form $T\cap U$, where $T$ is the set of all K3 surfaces with an integral model that reduces to a fixed model mod $p$, and $U \subset \KK_d$ is a Zariski open dense subvariety consisting of the complement of a finite union of Noether--Lefschetz divisors.  

We first explain the Zariski open $U \subset \KK_d$.  For $d=2$, following Elsenhans and Jahnel~\cite{elsenhans_jahnel_2008,elsenhans_jahnel:which}, we take $U = \KK_d \smallsetminus \KK_{2,1}^{-2}$ to be the complement of the tritangent locus; for $d=4$, following van Luijk~\cite{van_luijk_k3_2007}, we take $U = \KK_4 \smallsetminus \KK_{4,1}^{-2}$ to be the complement of the divisor of quartic K3 surfaces containing a line; for $d=6$, we take $U = \KK_d \smallsetminus \KK_{6,1}^{-2}$ to be the complement of the divisor of sextic K3 surfaces containing a line; and for $d=8$, we take $U$ to be the complement in $\KK_8$ of the union of $\KK_{8,2}^0$, $\KK_{8,1}^{-2}$, $\KK_{8,2}^{-2}$ (which defines the open locus $\KK_8^\circ$ where the discriminant K3 is defined and smooth, as in Proposition~\ref{prop:discriminantcontinuous}), as well as the complement of the preimage under $\KK_8^\circ \to \KK_2^\circ$ of the tritangent locus $\KK_{2,1}^{-2}$.

We now show that $T$ is Zariski dense in $\mathcal{K}_d$. In all cases, this holds by applying the following general result to the unirational varieties $\KK_d$ for $d \leq 8$.

\begin{lemma}
\label{lem:unirational}
Let $X$ be an equidimensional $\Z$-scheme with nonempty generic fiber $X_\Q$ of dimension $n$ and fix a $\Q$-point $x$ of $X$ reducing to an $\Fp$-point $x_p$ of the reduction $X_p$ modulo a prime $p$.  Assume that there is a dominant rational map $f : \P_\Z^n \dashrightarrow X$ such that $x_p=f(z_p)$ for some $\Fp$-point $z_p$ of $\P_{\Fp}^n$.  Then the set $T$ of $\Q$-points of $X_\Q$ that reduce to $x_p$ modulo $p$ is Zariski dense in $X_\Q$. 
\end{lemma}

We believe that this statement should follow from some weak
approximation type results for unirational varieties, though we could
not find an explicit reference.  Hence we include a direct proof of Lemma~\ref{lem:unirational}.

\begin{proof}
First we prove the result for $X=\P_\Z^n$. Suppose to the contrary that there is a closed proper subvariety $V\subset \P_\Z^n$ whose $\Q$-points contains $T$. We can assume, without loss of generality, that $V$ is codimension 1, i.e., $V=V(g)$ for some homogenous polynomial $g$ of degree $d$. For any prime $q$, the Schwartz--Zippel lemma (a weak form of the Lang--Weil bounds) then implies that the number of $\F_q$-points of $V_q$ is at most $dq^{n-1}$, an amount that becomes strictly less than $q^n + q^{n-1} + \cdots + 1$ as $q$ increases. Thus, for all but finitely many primes $q$, there exists at least one $\F_q$-point of $\P_{\F_q}^n$ that is not in $V_q$.  

Fix one such prime $q \neq p$, one such $\F_q$-point $y_q$ of $\P_{\F_q}^n$, and let $y$ be a lift of $y_q$ to a $\Q$-point of $\P_\Q^n$. Now consider the $\Q$-point $qx+py$ of $\P_\Q^n$, where we perform addition in some affine patch $\mathbb{A}_\Q^n$ containing $x$ and $y$.  Since $qx+py$ reduces to $x_p$ modulo $p$ we have that $qx+py\in T$ hence is a $\Q$-point of $V$.  However, this contradicts the fact that $qx+py$ reduces to $y_q$ modulo $q$, which by construction is not contained in $V_q$.  Hence $T$ must be Zariski dense in $\P_\Q^n$.

Now, assume we have a dominant rational map $f:\P_\Z^n\dashrightarrow X$ such that $x_p$ lifts to an $\F_p$-point $z_p$. By the above argument, the set $S$ of $\Q$-points of $\P_\Q^n$ that reduce to $z_p$ modulo $p$ is Zariski dense in $\P_\Q^n$. 
But since $f$ is dominant, the image of a dense subset is dense, so we obtain that $f(S)$ is Zariski dense in $X_\Q$.  Since $f(z_p)=x_p$ we have that $f(S)\subseteq T$, and hence that $T$ is Zariski dense in $X_\Q$ as well. 
\end{proof}

Finally, we can give a proof of our main result.

\begin{proof}[Proof of Theorem~\ref{thm:dense}]
For each $d \leq 8$, we construct a prime $p$ and a smooth K3 surface $X_p$ of degree $d$ over $\F_p$ with geometric Picard rank 2.  For $d=2$, this was first done by Elsenhans and Jahnel~\cite[Example~28(ii)]{elsenhans_jahnel_2008}, \cite[Example~6.1(ii)]{elsenhans_jahnel:which} in several different ways; for $d=4$, this was the original case done by van Luijk~\cite[\S~3]{van_luijk_k3_2007}; for $d=6$, we use the (reduction of the) K3 surface in Corollary~\ref{thm:explicitdeg6}; and for $d=8$, this was first done by Elsenhans and Jahnel~\cite[\S~8]{elsenhans_jahnel:which}, there is another example in \cite[\S5.4]{mckinnie_brauer_2017}, and we provide yet another example in Corollary~\ref{ex:deg8}.  In the cases $d \leq 6$, the geometric Picard group is generated by the polarization and a line, whereas in the case $d=8$, the isogenous discriminant K3 has geometric Picard group generated by the polarization and a line.  For $d \leq 6$, taking a lift to $\Q$ that is contained in the open $U=\KK_d \smallsetminus \KK_{d,1}^{-2}$, i.e., that does not contain a line, forces the lift to have geometric Picard rank 1 by Theorem~\ref{thm:genericdeg6}.  For $d=8$, we take a lift to $\Q$ whose isogenous discriminant K3 is contained in the open $U=\KK_2 \smallsetminus \KK_{2,1}^{-2}$, i.e., not containing a tritangent line, forcing geometric Picard rank 1 by Proposition~\ref{thm:disc_rank1}. 

Lemma~\ref{lem:unirational} ensures that the set of lifts $X$ to $\Q$ of $X_p$ is Zariski dense in $\KK_d$, since for $d \leq 8$, where the general polarized K3 surface is a complete intersection, we have that $\KK_d$ admits an integral model that is unirational over $\Z$, given as a quotient of a projective space by a linear algebraic group defined over $\Z$.  The intersection of this dense subset of $\KK_d(\Q)$ with the open subvariety $U$ is still dense in $\KK_d$ and consists of polarized K3 surfaces that are Noether--Lefschetz general.
\end{proof}

\bibliographystyle{amsplain}
\bibliography{k3pic1}

\end{document}